\documentclass[12pt]{article}
\usepackage{amsfonts}
\usepackage{amsmath}
\usepackage{amsthm}
\usepackage{graphicx}
\usepackage{hyperref}
\usepackage[latin1]{inputenc}
\usepackage[
backend=biber,
style=alphabetic,
sorting=ynt
]{biblatex}

\font\smallit=cmti10

\usepackage[printwatermark]{xwatermark}
\usepackage{xcolor}
\usepackage{graphicx}
\newtheorem*{theorem*}{Theorem}

\newtheorem{proposition}{Proposition}

\begin{document}

\begin{center}
{\bf THE MONDRIAN PUZZLE: A BOUND CONCERNING THE M(N) = 0 CASE}
\vskip 20pt
{\bf Cooper O'Kuhn}\\
{\smallit University of Florida, Gainesville, Florida, United States}\\
\vskip 10pt
{\bf Todd Fellman}\\
{\smallit Rensselaer Polytechnic Institute, Troy, New York, United States}\\
\end{center}
\vskip 30pt

\centerline{\bf Abstract}
\noindent
In response to the Numberphile video regarding the Mondrian Puzzle\footnote{https://www.youtube.com/watch?v=49KvZrioFB0}, we provide a lower bound on how many integers less than a given threshold $x$ satisfy $M(n) \neq 0$ where $M(n)$ is the quantity in which the Mondrian Puzzle is interested, i.e. the minimal difference in area between the largest and smallest rectangle in a set of incongruent, integer-sided rectangles which tile an $n$ by $n$ square. 

\vskip 30pt

\section*{\normalsize 1. Introduction}
	The Puzzle goes as follows: Dutch artist Piet Mondrian is prescribed by an art critic to, one, no longer make art works which contain congruent rectangles, two, only paint on square canvases, and three, only paint rectangles which have integer side lengths. This critic would then score Mondrian's work by how close in area all of the rectangles in the painting are. In other words, for a natural number $n \geq 3$, we let $M(n)$ denote the minimum value of this score for an $n$ by $n$ square which is the least possible difference between the largest and smallest area in a set of incongruent, integer-sided rectangles that tile an $n \times n$ square. 

\begin{figure}[h!]
  \begin{center}
       \includegraphics[width=0.38\linewidth]{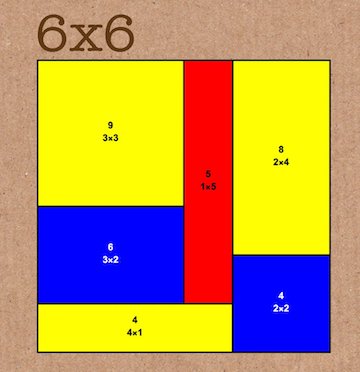}
  \end{center}
  \caption{An instance of Mondrian's Puzzle.}
\end{figure}

Figure 1, for instance, depicts a 6 by 6 canvas in which the largest rectangle has area 9 and the smallest rectangle has area 4, hence the score for the configuration is 5. In fact, this configurations achieves the lowest possible for a 6 by 6 canvas, so $M(6) = 5$.
	
The problem is generally interested in upper bounds for $M(n)$, but it has been asked whether there exists a natural number $n \geq 3$ such that $M(n) = 0$. Since the quantity $M(n)$ seems to generally increase (see [1]), such a number should not exist, and if proven, this would give a lower bound for $M(n)$. Typically, as is suggested by the proof methods of similar problems such as the squared square problem, i.e. whether there exists a square of integer side lengths that can be tiled by other smaller incongruent squares, one might try some sort of graph theoretic interpretation of the problem. Surprisingly, however, one can obtain bounds through the use of elementary number theoretic manipulations and injecting standard number theory results, namely those from Sieve Theory counting $z$-rough sets, subsets of the natural numbers in which all non-unitary divisors are greater than $z$. 

\begin{theorem*} We have
\begin{equation}
|\{n \leq x: M(n) \neq 0\}| \geq \left(\frac{e^{-\gamma}}{2} + o(1)\right)\frac{x}{\log{\log{x}}}
\end{equation}
\end{theorem*}

The manipulations in the proof of the main theorem are chiefly reductions of the set in (1) which deform it into a $z$-rough set. For this, we'll need the following proposition.

\begin{proposition} If $M(n) = 0$, there exists a proper divisor $d$ of $n^2$ such that one has the inequality 
\begin{equation}
d\tau(d) \geq n^2. 
\end{equation} 
\end{proposition}

\begin{proof} Suppose that $M(n) = 0$. Then there exists a set $S$ of rectangles which tile the $n$ by $n$ square in which all rectangles have the same area, say, $d$. We may assume $d \neq n^2$ as this trivially corresponds to the untiled square. Notice the identity 
\begin{equation}
d|S| = n^2
\end{equation}
which comes simply from adding the areas of the rectangles to obtain the area of the square. Thus, as all quantities in (3) are positive integers, we have that $d$ must divide $n^2$. Furthermore, since all the rectangles in $S$ have area $d$ and have integer side lengths as prescribed by the problem, the product of the side lengths of each rectangle must themselves be $d$ (as per the area of a rectangle); the number of these pairs, i.e. pairs of natural numbers whose product is $d$, is certainly bounded by $\tau(d)$, so
\begin{equation}
|S| \leq \tau(d)
\end{equation}
Suppose now that (2) holds for all proper divisors $d$ of $n^2$. We then have by (4)
$$n^2 = d|S| \leq d\tau(d).$$
\end{proof}

The remainder of the proof is to obtain a lower bound on the number of positive integers $n \leq x$ that don't have a proper divisor which satisfies (2). Suppose that 
\begin{equation}
d\tau(n^2) < n^2
\end{equation}
for all proper divisors $d$ of $n^2$. Surely, this contradicts (2) since $\tau(d) < \tau(n^2)$ as the divisors of $d$ are a subset of the divisors of $n^2$. Thus, we need only to count those natural numbers $\leq x$ which satisfy (5) (which we will see is an advantage). Similarly, $\tau(n^2) \leq \tau^2(n)$, so the even stronger condition 
\begin{equation}
d\tau^2(n) < n^2
\end{equation}
also contradicts (2). We make another similar but more complicated reduction. For any monotonic increasing function $f(x)$, let $I(x, f(x))$ denote the set of all natural numbers $n \leq x$ such that $\tau(n) > f(x)\log{x}$. Let $g(x)$ be some function tending to infinity slowly, and consider the set 
$I(x, g(x)\log{\log{x}})$. By the well-known asymptotic
$$\sum_{n \leq x} \tau(n) = (1 + o(1))x\log{x},$$
we have that 
$$|I(x, g(x)\log{\log{x}})| < (1 + o(1))\frac{x}{g(x)\log{\log{x}}} = o\left(\frac{x}{\log{\log{x}}}\right).$$
In other words, if we consider those $n \leq x$ which satisfy (6) but do \textit{not} belong to $I(x, g(x)\log{\log{x}})$, we will be skipping at most $o\left(\frac{x}{\log{\log{x}}}\right)$ of the $n \leq x$. Thus, we may assume $\tau(n) \leq g(x)\log{x}\log{\log{x}}$. To summarize, we wish to count those $n$ which satisfy (6) and $\tau(n) \leq g(x)\log{x}\log{\log{x}}$. Inserting this bound on $\tau$ into (6) gives 
$$d \cdot \left(g(x)\left(\log{x}\right)\left(\log{\log{x}}\right)\right)^2 < n^2$$ 
for all proper divisors $d$ of $n^2$. By the symmetry of the divisors ($d|n^2$ is equivalent to $\frac{n^2}{d}|n^2$), we obtain 
$$ \left(g(x)\left(\log{x}\right)\left(\log{\log{x}}\right)\right)^2 < d,$$ 
and the condition that $d$ properly divide $n^2$ now turns into $d > 1$. If one lets $z = \left(g(x)\left(\log{x}\right)\left(\log{\log{x}}\right)\right)^2$, one can say that $n^2$ is $z$-rough, or that $n$ is $z$-rough, since one can easily show by descent that $n^2$ being $z$-rough is equivalent to $n$ being $z$-rough. We can now use the Fundamental Lemma of the Selberg Sieve as in [2] to count the set of $z$-rough numbers: 
$$\{n \leq x: n\text{ is }z\text{-rough}\} \geq (e^{-\gamma} + o(1))\frac{x}{\log{z}} = \left(\frac{e^{-\gamma}}{2} + o(1)\right)\frac{x}{\log{\log{x}}},$$
which gives the theorem. 

\section*{\normalsize 2. The Way Forward} The above arguments seem truly maximal in that the bounds admit very little (if only constant) losses. Thus, a new idea is likely needed in order to gain a significantly better result. The following is a suggestion as to what that may look like. We begin by proving the following result.

\begin*{lemma} If 
\begin{equation}
|\{n \leq x: M(n) = 0\}| = o(x),
\end{equation}
then $M(n) \neq 0$ for all $n$.
\end*{lemma} 

\begin{proof} Suppose that $M(n) = 0$ for some positive integer and $n$ and that the set of exceptions had zero density. By the definition of $M(n)$ and scaling, one would also have that $M(kn) = 0$ for all positive integer $k$. This gives the set of exceptions $\leq x$ a cardinality of at least $\lfloor x/j\rfloor$.
\end{proof}

Hence, if the set of natural numbers $n$ for which $M(n) = 0$ has zero density, then no such natural numbers $n$ exist at all. Either this proof provides a path to a complete solution of lesser resistance or it illuminates the difficulty of achieving a bound like (7). 

Notice that the set on the LHS of (7) is the compliment of the set we were initially bounding, and thus our search for a lower bound now turns to a much more natural (though difficult) upper bound. Of course, with our previous method, the best we could achieve is a bound of the form 
$$|\{n \leq x: M(n) = 0\}| \leq x - o(x)$$
which is quite far from what is needed. The main issue arises from the extent to which geometric knowledge about the configurations in question is limited: one only deals with the areas of the rectangles upon proving the main theorem. However, the condition $M(n) = 0$ inevitably places many more constraints on a tiling than just the areas of the rectangles it contains, and thus improvements on this front which yield better bounds are conceivable. 

\vskip 150 pt

\section*{\normalsize Acknowledgements} The authors would like to acknowledge extremely helpful conversations with Kevin Ford, Carl Pomerance, and Jesse Thorner. 

\section*{\normalsize References}
\begin{enumerate}
\item (2020), The On-Line Encyclopedia of Integer Sequences, http://oeis.org/A276523
\item  Halberstam, H., and Richert, H. E. (2013). \textit{Sieve methods.} Courier Corporation.
\end{enumerate}

\end{document}